\documentclass[reqno,12pt]{amsart}

\usepackage{enumerate}
\usepackage{graphicx}
   \makeatletter
\@namedef{subjclassname@2010}{%
  \textup{2010} Mathematics Subject Classification}
\makeatother
\usepackage{amsmath,amssymb,amsthm}
\usepackage{cite}
\newcommand{\X}{{\bf X}}
\newcommand{\bc}{{\bf c}}

\newcommand{\R}{\mathbb{R}}
\newcommand{\Q}{\mathbb{Q}}

\newcommand{\C}{\mathbb{C}}
\newcommand{\Z}{\mathbb{Z}}

\newcommand{\F}{\mathbb{F}}

\makeatletter
\newcommand{\lang}{\langle}
\newcommand{\rang}{\rangle}

\newcommand{\x}{{\bf x}}
\newcommand{\bs}{{\bf s}}
\newcommand{\bm}{{\bf m}}

\newcommand{\leg}[2]{{#1\overwithdelims () #2}}
\DeclareMathOperator{\ord}{ord}

\newtheorem{thm}{Theorem}[section]

\newtheorem{lem}[thm]{Lemma}
\newtheorem*{lem*}{Lemma}
\newtheorem{prop}[thm]{Proposition}

\theoremstyle{remark}
\newtheorem*{rem*}{Remark}

\theoremstyle{definition}

\date{\today}
\author{Holley Friedlander}
\address{Dickinson College}
\email{friedlah@dickinson.edu}
\subjclass[2010]{Primary 11F68, 11M41; Secondary 11R58, 11M41}
\keywords{Weyl group multiple Dirichlet series, rational function field}
\title[multiple Dirichlet series over the rational function field]{Twisted Weyl group multiple Dirichlet series over the rational function field}

\begin{document}

\begin{abstract}
Weyl group multiple Dirichlet series are Dirichlet series in $r$ complex variables, with analytic continuation to $\C^r$ and a group of functional equations isomorphic to the Weyl group of a reduced root system of rank $r$. Such series may be defined for any global field $K$, but in the case when $K$ is an algebraic function field they are expected to be, up to a variable change, rational functions in several variables.  We verify the rationality of these functions in the case when $K=\F_q(T)$, and describe the denominators and support of the numerators.

\end{abstract}

\maketitle


\section{Introduction}
\label{sec:introduction}
\subsection{Overview}
To $n$ a positive integer, $K$ a global field containing the $2n$th roots of unity, and $\Phi$ a reduced root system of rank $r$, one may associate a Weyl group multiple Dirichlet series. These objects are Dirichlet series in $r$ complex variables, with analytic continuation to $\C^r$ and a group of functional equations isomorphic to the Weyl group of $\Phi$. They arise, for example, as Fourier-Whittaker coefficients of Eisenstien series on metaplectic groups and have applications in analytic number theory \cite{C,hoffstein,patterson,GH,mdssurvey}.   

In this paper we consider the case when $K=\F_q(T)$ is the field of rational functions over the finite field $\F_q$ with $q$ elements. In this case --- in fact, if $K$ is any algebraic function field --- Weyl group multiple Dirichlet series are expected to be rational functions in several variables. The first one to note the rationality of these series was Hoffstein, who in \cite{hoffsteina2} obtained a Weyl group multiple Dirichlet series of type $\Phi=A_2$ as the Mellin transform of a half-integral weight Eisenstein series. In \cite{FF}, Fisher and Friedberg also obtain an analytic continuation for and prove the rationality of type $A_2$ series (and provide explicit examples for $K$ the rational function field {\em and} the function field associated to a specific genus one curve --- for similar $A_3$ examples, see \cite{FF2}; for more $A_2$ and $K$ genus one, see \cite{ians}), but with a different construction that avoids the use of metaplectic Eisenstien series. Our methods are more in line with the latter construction. However, in the time since that publication, a more cohesive theory of Weyl group multiple Dirichlet series has developed, cf.~\cite{wmds1,stable,wmds3,wmds4,wmnds}. For $A_2$ and $A_3$ examples illustrative of this newer approach, see \cite{C}; for $A_r$ and $n \gg r$, see \cite{chintamohler}. 

Our work expands on the previous literature in two ways. First, we work with $\Phi$ and $n$ arbitrary. Second, we consider ``twisted'' Weyl group multiple Dirichlet series over the rational function field. Twisted series are essentially twists of the original series by $n$th order characters. In this sense, one may think of twisted Weyl group multiple Dirichlet series as analogue to $L$-functions. 
Just as in the case of the zeta function associated to an algebraic function field, we expect the coefficients of these twisted series to encode information about the arithmetic of the defining curve. The first step is to determine the support of the series, we which we address in Theorem \ref{thm:localglobal}. Specifically, we prove that such series are in fact rational functions and show that after a variable change, Weyl group multiple Dirichlet series over the rational function field can be expressed as finite weighted sums of ``local'' series, which act analogue to Euler factors in the construction of the global object. This result generalizes observations from \cite{chintamohler,C} that note a similarity between the local and global series in certain special cases. It also provides a roadmap to compute families of examples. 

%
%

\subsection{Series construction}
To better understand Theorem \ref{thm:localglobal}, it is helpful to have an idea of how one constructs Weyl group multiple Dirichlet series in the rational function field case. We give a brief overview here, with full details appearing in Section \ref{sec:prelim}. Let $n \geq 1$ be an integer such that $q \equiv 1 \mod{2n}$. Let $\mathcal{O}=\F_q[T]$ and let $\mathcal{O}_{mon}\subset\mathcal{O}$ denote the set of monic polynomials. For $f \in \mathcal{O}$, the norm of $f$ is $|f|:=q^{\deg f}$. Choose $r$-tuples ${\bf s}=(s_1,\ldots, s_r)$ of complex variables and ${\bf m}=(m_1,\ldots,m_r)$ of elements in $\mathcal{O}^\times$. We call $\bm$ a twisting parameter. When $\bm=(1,\ldots,1)$, we say that the series is untwisted. Now define the degree $n$ Weyl group multiple Dirichlet series of type $\Phi$ with twisting parameter $\bm$ by
\begin{equation}
\label{eqn:zdef1}
Z({\bf s};{\bf m},\Phi,n,\F_q(T))=Z({\bf s};{\bf m}):=\sum_{{\bf c}\in (\mathcal{O}_{mon})^r} \frac{H({\bf c};{\bf m})}{|c_1|^{s_1}\cdots |c_r|^{s_r}}.
\end{equation} 

For fixed $p \in \mathcal{O}_{mon}$ irreducible and $\ell=(l_1,\ldots,l_r) \in (\Z_{\geq 0})^r$, the $p$-part of $Z({\bf s};{\bf m})$ is a generating function in $|p|^{-s_1},\ldots, |p|^{-s_r}$ for the $p$-power coefficients $H(p^{k_1},\ldots,p^{k_r};p^{l_1},\ldots,p^{l_r})$. Although $Z(\bs;\bm)$ is {\em not} Eulerian in general, the $p$-parts completely determine the coefficients $H(c_1,\ldots,c_r;{\bf m})$ via a twisted multiplicativity relation that acts analogue to an Euler product (see Section \ref{subsec:gsc}). 

The $p$-parts are built out of Gauss sums using combinatorial data from $\Phi$. We will follow the Chinta--Gunnells construction (see Section \ref{sec:cgconstruction}), which defines the $p$-parts via an averaging technique analogous to the Weyl character formula. This method yields global series with the desired analytic properties of analytic continuation and Weyl group of functional equations for all $\Phi$ and $n$ \cite{wmnds}. There are other methods to define the $p$-parts, notably the crystal graph technique of Brubaker, Bump, and Friedberg \cite{wmds1,wmds2,wmds4,typea}. Equivalence of the Brubaker--Bump--Friedberg and Chinta--Gunnells constructions has been shown in several cases \cite{ppart,me2,mcnamara,chintaoffen}.

As previously mentioned, the series $Z({\bf s};\bm)$ are expected to be rational in $q^{-s_1}, \ldots, q^{-s_r}.$ It was noticed in \cite{chintamohler,C} that in the untwisted case when $\Phi$ is type $A$ and $n \gg r$, up to a rational factor, a simple change of variables transforms the $p$-parts into the global series. This correspondence parallels the relationship between the Euler factors of the zeta function associated to the projective line and the global series. Indeed, Chinta and Gunnells \cite{wmnds} use this correspondence as a basis for their combinatorial construction of the $p$-parts. 
Our results generalize previous case-by-case observations: first, Theorem \ref{thm:rat} proves that the global series are indeed rational functions. Proposition \ref{prop:localglobal} then shows that in the untwisted case, a simple change of variables transforms a modified $p$-part into the global series, and Theorem \ref{thm:localglobal} shows that in the twisted case, this same variable change allows us to express the the global series as a finite weighted sum of modified $p$-parts. (The modified $p$-parts are simply $p$-parts multiplied by a prefixed rational factor.) Moreover, after normalizing by a product of zeta functions, the weights are coefficients of the original series $Z(\bs;\bm)$.

This paper is organized as follows. In Section \ref{sec:prelim} we introduce relevant notation and review the Chinta--Gunnells construction of the $p$-parts and their combinatorial properties \cite{ppart,me2,wmnds}. In Section \ref{sec:GFE} we use the functional equations stated in \cite{wmnds} for $K$ any number field to derive explicit functional equations for the case when $K=\F_q(T)$. In Section \ref{sec:global}, we prove the rationality of the series and describe the relationship between the $p$-parts and the global series in detail. In particular, Propositions \ref{prop:localglobal} and Theorem \ref{thm:localglobal} completely characterize Weyl group multiple Dirichlet series associated to the rational function field in terms of local $p$-parts. We also provide two low-rank examples.

\subsection*{Acknowledgements}
Thank you to Paul Gunnells, who advised the thesis that led to this paper, for his support and encouragement over several years of work on this project. Thanks also to Gautam Chinta for several very helpful conversations. I also had extremely valuable discussions with Ben Brubaker, Solomon Friedberg, Daniel Bump, Jeff Hoffstein, Anna Pusk\'{a}s, Ian Whitehead, and TingFang Lee; thank you.


  \section{Preliminaries}
    \label{sec:prelim}

  Fix an integer $n\geq 1$ and let $q=p^k$, for $p$ a rational prime and $k$ a positive integer, be such that $q \equiv 1 \mod 2n$. We consider $K=\F_q(T)$ the field of rational functions in $T$ over the finite field $\F_q$ with $q$ elements,  where $\mathcal{O}=\F_q[T]$ denotes the associated polynomial ring and $\mathcal{O}_{mon}\subset \mathcal{O}$ the subset of monic polynomials.  The completion of $K$ at the place corresponding to $\pi_\infty=T^{-1}$ is the field $K_\infty=\F_q((T^{-1}))$ of Laurent series in $\pi_\infty$. We have $\mathcal{O} \subset K \subset K_\infty$. For $f=\sum_{i=-k}^\infty a_i\pi_\infty^i \in K_\infty$, the degree of $f$ is the smallest $i$ such that $a_i \neq 0$, and the
norm $|\cdot|: \mathcal{O} \to \C^\times$ is defined by $|f|:=q^{\deg f}$.  For the remainder of the text, we assume that $P$ is a prime element of $K$, i.e.~$P\in \mathcal{O}_{mon}$ is irreducible.  
     
  \subsection{Gauss sums}
      The coefficients of Weyl group multiple Dirichlet series involve $n$th order Gauss sums. Before defining these generalized Gauss sums, we first recall the traditional finite field Gauss sum definition. Let $\mu_n=\{a \in \F_q:a^n=1\}$ be the $n$th roots of unity in $\F_q^\times$, and fix an embedding $\epsilon:\mu_n \to \C$.
Consider the multiplicative character $\chi:\F_q^\times \to \mu_n$ defined by $a \mapsto a^{(q-1)/n}$ and an additive character $e_0:\F_p^\times \to \C$, which we will take as $a\mapsto \exp{2\pi i a/p}$. To extend $e_0$ to $\F_q$, put $e_\ast=e_0\circ Tr_{\F_q/\F_p}$. For $t \in \Z$, define the Gauss sum
    \begin{equation}
    \label{eqn:gausssum}
    \tau(\epsilon^t)=\sum_{a\in \F_q^\times}\epsilon(\chi(a))^te_\ast(a).
    \end{equation}
A detailed listing of the properties of these sums see can be found in \cite[Section 8.2]{irelandrosen}. In particular, one can show that $\tau(\epsilon^t)\tau(\epsilon^{-t})=q.$
     
We now extend the notion of Gauss sum to $\mathcal{O}$. As we will no longer need to refer explicitly to the characteristic of $K$, from now on we let $p:=|P|$. For $a \in K_\infty$, write $a=\sum_{i\geq-N} a_{i}\pi_\infty^i$ and let $\psi(a)=a_{-1}$ and $\psi^\ast(a)=\psi(T^2a)$. Define a global additive character $e=e_\ast\circ\psi^\ast$, where $e_\ast$ is the additive character on $\F_q$ defined above. One can show that $e$ satisfies 
    \(
    \mathcal{O}=\{x \in K_\infty: e(x\mathcal{O})=1\}.
    \) 
    Let $m,c \in \mathcal{O}$. For $t \in \Z$, define the generalized Gauss sum
    \begin{equation}
    \label{eqn:globgauss}
    g(m,c;\epsilon^t)=g_t(m,c):=\sum_{y \not\equiv 0 \mod c} \epsilon\Bigg(\leg{y}{c}_n\Bigg)^te(my/c),
    \end{equation}
    where $\displaystyle{\leg{y}{c}_n}$ is the $n$th order residue symbol for $K$. We note that for $a,b \in \mathcal{O}_{mon},$  the $n$th order reciprocity law, cf.~\cite[Theorem 3.5]{rosen} and our assumption $q \equiv 1 \mod {2n}$ imply 
    \[
    \leg{a}{b}_n=\leg{b}{a}_n
    .\]
    
 After an appropriate identification of the residue field $\left(\mathcal{O}/c\mathcal{O}\right)^\times$ of $c$ with the finite field $\F_{q^{\deg c}}$, one sees that $g_t(1,c)$ corresponds to the $\F_{q^{\deg c}}$ Gauss sum $\tau(\epsilon^t)$.
    Accordingly, the following properties of generalized Gauss sums follow directly from the corresponding properties of finite field Gauss sums \cite{Kubota, hoffstein, C}, cf.~\cite[Section 8.2]{irelandrosen}.
    \begin{enumerate}
    \item If $(c,c')=1$, we have 
    \[
    g_t(m,cc')=\leg{c}{c'}_n^t\leg{c'}{c}_n^tg_t(m,c)g_t(m,c').
    \]
    \item If $(a,c)=1$, we have
    \[
    g_t(am,c)=\leg{a}{c}_n^{-t}g_t(m,c).
    \]
    \item Let $k,l \in \Z_{\geq 0}$, and let $\phi(P^l)=$ be the number of elements of $(\mathcal{O}/P^l\mathcal{O})^\times$. Then
    \begin{equation}
    \label{eqn:gauss}
    g_t(P^l,P^k)=\left\{\begin{array}{cl}p^lg_{tk}(1,P)&\mbox{if } k=l+1;\\
    \phi(P^k)&\mbox{if } n|tk \mbox{ and } l \geq k;\\
    0&\mbox{otherwise}.\end{array}\right.
    \end{equation}
    \item If $(t,n)=1$, then 
    \[
    g_t(1,P)g_{-t}(1,P)=|P|=p.
    \]
    \end{enumerate}
    In what follows, we denote $g_t(1,c)$ by $g_t(c)$.

%

\subsection{Roots and weights}
The $p$-parts are built using combinatorial data from an irreducible, reduced root system $\Phi$ of rank $r$, see \cite[Chapter 9]{humphreys}. Any root of $\Phi$ can be written as a $\Z$-linear combination of the simple roots $\alpha_1,\ldots, \alpha_r$. We say $\alpha \in \Phi$ is positive (negative) if when written as a sum of simple roots $\alpha=\sum k_i \alpha_i$, all $k_i$ are nonnegative (nonpositive). We have a decomposition $\Phi=\Phi^+\cup\Phi^-$ into positive and negative roots.  If $\alpha \in \Phi^+$, we will write $\alpha>0$. Let $\Lambda$ be the root lattice of $\Phi$, i.e.~the $\Z$-span of the simple roots. Define the generalized height function $d:\Lambda \to \Z$ by
    \begin{equation}
    \label{eqn:height}
    d:\lambda=\sum_{i=1}^r \lambda_i\alpha_i \mapsto \sum \lambda_i.
    \end{equation}
    
    Let $W$ be the Weyl group of $\Phi$, and fix a $W$-invariant symmetric, bilinear, positive definite inner product $(\cdot,\cdot)$ on $\Lambda \otimes \R$ normalized such that the short roots all have length one. This implies that for any $\alpha,\beta \in \Lambda$, we have $(\alpha,\beta) \in \frac{1}{2}\Z$. In particular 
    \[
    \|\alpha\|^2=\left\{\begin{array}{ll} 1&\mbox{for all $\alpha$ in types } A,D,E,\\
    1&\mbox{for $\alpha$ a short root in types } B,C,F_4,G_2,\\
    2&\mbox{for $\alpha$ a long root in types }B,C,F_4,\\
    3&\mbox{for $\alpha$ a long root in type }G_2.
   \end{array}\right.\]
The Weyl group $W$ is generated by the simple reflections
    \begin{equation}
    \label{eqn:sigj}
    \sigma_j(\alpha_i)=\alpha_i-\lang\alpha_i,\alpha_j\rang\alpha_j,
    \end{equation}
where we define $\lang \alpha_i,\alpha_j \rang:=2\frac{(\alpha_i,\alpha_j)}{(\alpha_j,\alpha_j)}$. We will denote the $ij$-entry of the Cartan matrix by $c(i,j)=\lang\alpha_i,\alpha_j\rang$.  For $w \in W,$ let $l(w)$ be the number of $\sigma_j$ in any reduced expression for $w$, and put $sgn(w)=(-1)^{l(w)}$.

    Define the simple coroots $\check{\alpha}_i=2\alpha_i/(\alpha_i,\alpha_i)$ for $i=1,\ldots r$. The fundamental weights $\{\varpi_1,\ldots,\varpi_r\}$ of $\Phi$ are the corresponding dual basis with respect to $(\cdot,\cdot)$. In particular, we have $\lang \varpi_i,\alpha_j\rang=\delta_{ij}$ and may express each simple root $\alpha_i=\sum c(i,j)\varpi_j$  as a linear combination of fundamental weights. Let $L$ be the weight lattice of $\Phi$ generated by the fundamental weights. There is a partial order on $L$: we say $\mu \succeq \xi$ if $\mu-\xi=\sum k_i\varpi_i$ with all $k_i$ nonnegative. We say $\mu \in L$ is dominant if $\lang \mu,\alpha_i\rang \geq 0$ for all $i=1,\ldots, r$ and regular dominant if the inequality is strict. For example, $\rho=\sum_{i=1}^r \varpi_i$ is a regular dominant weight.

    \subsection{Chinta--Gunnells construction}
    \label{sec:cgconstruction}
       Recall that $P$ is a prime of $K=\F_q({T})$ of norm $p=q^{\deg{P}}$. Here we describe the Chinta--Gunnells method \cite{wmnds} to construct, for each prime $P$, generating functions in $p^{-s_1},\ldots, p^{-s_r}$ for the $P$-power coefficients of $Z(\bs;\bm)$. These generating functions are called $p$-parts. This construction involves a technique analogous to the Weyl character formula to define an invariant rational function by averaging over the Weyl group. 

We first define an appropriate Weyl group action on rational functions. Fix an $r$-tuple of nonnegative integers $\ell=(l_1,\ldots,l_r)$. The $\ell$ we choose will be determined by $P$ and the twisting parameter $\bm$; hence, we also call $\ell$ a twisting parameter. The tuple $\ell$ determines a regular dominant weight
\begin{equation}
\label{eqn:theta}
\theta=\theta(\ell):=\sum_{i=1}^r (l_i+1)\varpi_i
\end{equation}
and a $W$-action on $\Lambda$:
\begin{equation}
\label{eqn:bullet}
w\bullet\lambda:=w(\lambda-\theta)+\theta.
\end{equation}
Note that when $w=\sigma_j$ is a simple reflection, we have 
\(
\sigma_j\bullet\lambda=\sigma_j\lambda+(l_j+1)\alpha_j.
\)

Consider $A=\C[\Lambda]$, the ring of Laurent polynomials on $\Lambda$. The ring $A$ consists of all expressions $f$ of the form $f=\sum_{\beta \in \Lambda} c_\beta \x^\beta$ with $c_\beta \in \C$ almost all zero. Multiplication in $A$ is defined by addition in $\Lambda$: $\x^\beta\x^\lambda=\x^{\lambda+\beta}$ and we identify $A$ with $\C[x_1,x_1^{-1}\ldots,x_r,x_r^{-1}]$ via $\x^{\alpha_i}\mapsto x_i$. 

Our goal is to define a Weyl group action on the field of fractions $\tilde{A}$ of $A$. First, define a change of variables action on $A$ by 
\begin{equation}
\label{eqn:cov}
\left(\sigma_j(\x)\right)_i=p^{-c_{ij}}x_ix_j^{-c_{ij}}.
\end{equation}
This action is essentially a reformulation of the standard action of $W$ on $\Lambda$. One can check that if $f_\beta(\x)=\x^\beta$ is a monomial, then
\(
f_\beta(w\x)=p^{d(w^{-1}\beta-\beta)}\x^{w^{-1}\beta}.
\)

For $\alpha \in \Phi$, let
\(
n(\alpha)=n/\gcd(n,\|\alpha\|^2)\). Consider the sublattice $\Lambda'\subset \Lambda$ generated by the set $\{n(\alpha)\alpha\}_{\alpha \in \Phi}$. Define $\tilde{A}_\lambda$ as the set of functions $f/g \in \tilde{A}$ such that $g$ lies in the kernel of the map $\nu:\Lambda \to \Lambda/\Lambda'$ and $\nu$ maps the support of $f$ to $\lambda$. Then we have the decomposition
\[
\tilde{A}=\bigoplus_{\lambda \in \Lambda/\Lambda'}\tilde{A}_\lambda.
\]
For $f(\x) \in A_\beta$, the Chinta-Gunnells action associated to the simple reflection $\sigma_k\in W$ is 
\begin{equation}
\label{eqn:act}
(f|_\ell \sigma_k)(\x)=(\mathcal{P}_{\beta,\ell,k}(x_k)+\mathcal{Q}_{\sigma\bullet\beta,\ell,k}(x_k))f(\sigma_k\x),
\end{equation}
where $\mathcal{P}$ and $\mathcal{Q}$ are rational functions defined below in the following way: fix $k \in \{1,\ldots, r\}$ and for $\lambda \in \Lambda$, define \(\delta_{\ell,k}=\delta_k(\lambda):=d(\sigma_k\bullet\lambda-\lambda)\). For positive integers $a$ and $m$, let $(a)_m:=a-m\lfloor a/m\rfloor$ be the remainder of $a$ upon division by $m$. We put $(-a)_m=0$ if $(a)_m=0$ and $(-a)_m=m-(a)_m$ otherwise.  Then
    \begin{align}
    \label{eqn:locPQ}
    \mathcal{P}_{\beta,\ell,k}(x_k)&=(p x_k)^{l_k+1-(\delta_k(\beta))_{n(\alpha_k)}}\frac{1-1/p}{1-(p x_k)^{n(\alpha_k)}/p},\\
    \label{eqn:locQ}
    \mathcal{Q}_{\beta,\ell,k}(x_k)&=-g^\ast_{-\|\alpha_k\|^2\delta_k(\beta)}(P)(p x_k)^{l_k+1-{n(\alpha_k)}}\frac{1-(p x_k)^{n(\alpha_k)}}{1-(p x_k)^{n(\alpha_k)}/p},\nonumber
    \end{align}
     where $g^\ast_t(P)$ denotes the normalized Gauss sum
    \begin{equation}\label{gdef}
    g^\ast_t(P)=\left\{\begin{array}{ll}-1 &\mbox{ if } t\equiv 0\mod{n},\\
    g_t(P)/p &\mbox{ otherwise.}\end{array}\right.\end{equation}
Note that action \eqref{eqn:act} satisfies the defining relations for $W$ \cite[Theorem 3.2]{wmnds} and hence it extends to all of $W$.

We are now ready to define the $p$-parts in terms of a $W$-invariant rational function. Let $j(w,\x)=sgn(w)\prod_{\alpha \in \Phi(w)}\x^\alpha$, $\Delta(\x)=\prod_{\alpha>0} (1-p^{n(\alpha)d(\alpha)}\x^{n(\alpha)\alpha})$, and $D(\x)=\prod_{\alpha>0} (1-p^{n(\alpha)d(\alpha)-1}\x^{n(\alpha)\alpha})$. Then
\begin{equation}
\label{eqn:ppart}
F(\x;\ell):=\frac{1}{\Delta(x)}\sum_{w \in W} j(w,\x)(1|_{\ell}w)(\x),
\end{equation}
is invariant under \eqref{eqn:act} and $F(\x,\ell)D(\x)$ is polynomial in the $x_i=p^{-s_i}$ \cite[Theorem 3.5]{wmnds}.
For each $\ell$, the {\em $p$-part} is the polynomial $N(\x;\ell):=F(\x,\ell)D(\x)$. The $\ell$ we choose is determined by $P$ and $\bm$ in the following way: let $l_i$ be the largest nonnegative integer such that $P^{l_i}$ divides $m_i$, denoted $P^{l_i}||m_i$. 
Then the coefficient 
    \(H(P^{k_1},\ldots, P^{k_r};P^{l_1},\ldots, P^{l_r})
        \) of $Z(\bs;\bm)$ is the $x_1^{k_1}\cdots x_r^{k_r}$ coefficient of $N(\x;\ell)$. In other words, the $p$-parts are generating functions for the $P$-power coefficients of $Z(\bs;\bm)$.

  We remark that the invariance of $F(\x;\ell)$ under the action \eqref{eqn:act} yields a functional equation. Write $F(\x,\ell)=\sum_{\beta \in \Lambda/\Lambda'} f_\beta(\x)$ so that $f_\beta(\x) \in \tilde{A}_\beta$. Then
   \(
   F(\x;\ell)=\sum_{\beta \in \Lambda/\Lambda'}(\mathcal{P}_{\beta,\ell,k}(x_k)+\mathcal{Q}_{\sigma\bullet\beta,\ell,k}(x_k))f_\beta(\sigma_k\x).
   \)
  One checks that $\mathcal{P}_{\beta,\ell,k}(x_k)f_\beta(\sigma_k\x)\in \tilde{A}_\beta$ and $\mathcal{Q}_{\sigma\bullet\beta,\ell,k}(x_k)f_\beta(\sigma_k\x) \in \tilde{A}_{\sigma_k\bullet \beta}$. Since $\sigma_k\bullet$ is an involution, it follows from \cite[Theorem 3.5]{wmnds} that
   \begin{equation}
   \label{eqn:ppfe}
   F_\beta(\x)=\mathcal{P}_{\beta,\ell,k}(x_k)F_\beta(\sigma_k\x)+\mathcal{Q}_{\beta,\ell,k}(x_k)F_{\sigma_k\bullet\beta}(\sigma_k\x).
   \end{equation}

  \subsection{Properties of $p$-parts}
  \label{subsec:ppart}
We will briefly review some structural properties of the $p$-parts. Recall that $\theta=\sum_{i=1}^r (l_i+1)\varpi_i$. Write $N(\x;\theta)=N(\x;\ell)=\sum a_\lambda \x^\lambda$, where it is understood that $P$ is fixed and for $\lambda={\sum_{i=1}^r k_i\alpha_i}$, we have $a_\lambda=H(P^{k_1},\ldots, P^{k_r};P^{l_1},\ldots, P^{l_r})$.
 Let $\Pi_\theta$ be the convex hull of $\theta-w\theta$ for $w \in W$. More precisely, $\Pi_\theta$ is the weight polytope for the irreducible representation of {\em lowest} weight $-\theta$, shifted by $\theta$. If $\Theta$ is the set of dominant weights in the representation of {\em highest} weight $\theta$, then all points of $\Pi_\theta$ have the form $\theta-w\xi$ where $w \in W$ and $\xi \in \Theta$. Proof of the following results may be found in \cite{me2}.

%

  

\begin{thm}
\label{thm:support}
The support of $N(\x;\theta)$ is contained in $\Pi_\theta$. \qed
\end{thm} 
We will refer to the coefficients associated to the vertices of $\Pi_\theta$ as stable coefficients. When $n \gg r$, these are the only nonzero coefficients. The stable coefficients are completely determined by recurrence relations on the coefficients of the $p$-parts (after setting the constant term to be 1). The unstable coefficients (those strictly inside of the polytope $\Pi_\theta$) are also determined by the recurrence relations, as we see in the following theorem.

\begin{thm}
\label{thm:unstab}
Define $\Theta^+$ to be the set of all regular dominant weights in the representation of highest weight $\theta$.  Then the $p$-part $N(\x;\theta)$ is completely determined up to the coefficients $a_{\theta-\xi}$ for $\xi \in \Theta^+$. \qed
\end{thm}

In fact, the proof of \cite[Theorem 5.4]{me2} shows that any polynomial satisfying the recurrence relations of $N(\x;\theta)$ can be written as a sum of shifted $p$-parts. 

\begin{thm} 
\label{thm:unstabsum} Define $\Theta^+$ to be the set of all regular dominant weights in the representation of highest weight $\theta$. Suppose that $\mathcal{N}(\x)$ is any polynomial satisfying the recurrence relations of $N(\x;\theta)$. Then
\[\mathcal{N}(\x)=\sum_{\xi \in \Theta^+} m_\xi N(\x;\xi)\xi^{\theta-\xi},\]
where $m_{\xi}$ are complex numbers. \qed

\end{thm}
Finally, multiplying the $p$-parts by the rational factor $\Delta(\x)/D(\x)$ reflects the support over $\Pi_\theta$. 
\begin{thm}
\label{thm:gap}
The support of $\Delta(\x)N(\x;\theta)$ lies outside the polytope $\Pi_\theta$. \qed
\end{thm}

As we shall see, these modified $p$-parts are related to the global series up to a variable change. 
\subsection{Global series coefficients}
\label{subsec:gsc}
Let us explain how to define the coefficients of the global series $Z(\bs;\bm)$. 
Recall that the $p$-parts are generating functions for the coefficients $H(P^{k_1},\ldots,P^{k_r};P^{l_1},\ldots,P^{l_r})$. 
        Up to a product of residue symbols, the coefficient $H(\bc;\bm)$ is the product over all $P$ dividing the $c_i$ of $H(P^{k_1},\ldots P^{k_r};P^{l_1},\ldots,P^{l_r})$, where $k_i$ and $l_i$ satisfy $P^{k_i}||c_i$ and $P^{l_i}||m_i$. 
        The residue symbols required follow from the twisted multiplicativity relation:
        for fixed $(c_1\cdots c_r,c_1'\cdots c_r')=1$, put
   	\[
   	H(c_1c_1',\ldots,c_rc_r';\bm)=H(c_1,\ldots,c_r;\bm)H(c_1',\ldots,c_r';\bm)\varphi({\bf c};{\bf c}'),
   	\]
   	where 
   	\[
   	\varphi({\bf c};{\bf c}')=\prod_{i=1}^r \leg{c_i}{c_i'}^{\|\alpha_i\|^2}\leg{c_i'}{c_i}^{\|\alpha_i\|^2}
   	\prod_{i<j}\leg{c_i}{c_j'}^{2c(i,j)}\prod_{i<j}\leg{c_i'}{c_j}^{2c(i,j)}.
   	\]
   	In addition, if $(c_1\cdots c_r;m_1'\cdots m_r')=1$, then
   	\[
   	H(c_1,\ldots,c_r;m_1m_1',\ldots,m_rm_r')=\prod_{j=1}^r \leg{m_j'}{c_j}^{-\|\alpha_j\|^2}H(c_1,\ldots,c_r;m_1,\ldots,m_r).
   	\]
   	
   	With this definition of $H(\bc;\bm)$, we can make definition \eqref{eqn:zdef1} explicit. Define the degree $n$ Weyl group multiple Dirichlet series of type $\Phi$ over $\F_q(T)$ with twisting parameter $\bm$ by
   	\begin{equation}
   	 \label{eqn:zdef}
   		Z({\bf s};{\bf m}):=\sum_{{\bf c}\in \mathcal{O}_{mon}^r} \frac{H({\bf c};{\bf m})}{|c_1|^{s_1}\cdots |c_r|^{s_r}}.
   	 \end{equation}
Also define the normalized series
   	 \(
   	    		Z^\ast({\bf s};{\bf m})=\Xi(\bs) Z(\bs;\bm),
   	\)
where
   	 \begin{equation}
   	 \label{eqn:xi}
   	 \Xi(\bs):=\prod_{\alpha=\sum k_i\alpha_i>0} \zeta_{\mathcal{O}}(1+n(\alpha)\sum_{i=1}^r k_i(s_i-1))
   	  \end{equation}
   	  and $\zeta_{\mathcal{O}}$ denotes the zeta function 
	  \(\zeta_{\mathcal{O}}(s):=\sum_{c \in \mathcal{O}_{mon}}|c|^{-s}=(1-q^{1-s})^{-1}.\) Note that like $Z(\bs;\bm)$, the product $\Xi(\bs)$ depends on $\Phi$ and $n$.

\section{Global Functional Equations}
\label{sec:GFE}
In this section we derive the functional equations of Weyl group multiple Dirichlet series defined over $\F_q(T)$. Our main reference is \cite{wmnds}, which derives the functional equations of $Z^\ast(\bs;\bm,K,\Phi,n)$ for $K$ any number field. Although it is understood that their arguments apply to $K$ any global field, we require explicit functional equations to discern the structural properties of the global series in Section \ref{sec:global} and thus, we include the specific details of the rational function field argument here.

    To state the functional equations of $Z^\ast(\bs;\bm)$, we first define a slightly more general class of series. Let $I=(I_1,\ldots,I_r)$ be an $r$-tuple of integers such that $I_j \in \{0,\ldots, n(\alpha_j)-1\}$. Then define
    \[Z(\bs;\bm,I):=\sum_{\stackrel{\bc \in \mathcal{O}_{mon}}{\deg c_j \equiv I_j \mod n(\alpha_j)}} \frac{H(\bc;\bm)}{|c_1|^{s_1}\cdots|c_r|^{s_r}}\]
       		        and 
    \[Z^\ast(\bs;\bm,I)=\Xi(\bs)Z(\bs;\bm,I).\]

Fix a simple reflection $\sigma_i \in W$. Define an action of $\sigma_i$ on the $r$-tuple $\bs$ by 
\[
(\sigma_i\bs)_j=s_j-c(j,i)(s_i-1).\] This action is consistent with the change of variables action \eqref{eqn:cov} on monomials. For $\bm$ and $I$ fixed, let $J_i(m,I)=\deg m_i-\sum_{j \neq i} c(j,i)I_j$, and let $\sigma_i\bullet I$ be the tuple whose $j$th entry is the $\alpha_j$ coefficient of $\sigma_i\bullet \big( \sum_{i=1}^r I_j\alpha_j\big)$. 
We will show that
    	\begin{align}
       		\label{eqn:globfe}
       		Z^\ast(\bs;\bm,I)&=|m_i|^{1-s_i} P_{I_i,J_i(m,I)}^{\|\alpha_i\|^2}(s_i)Z^\ast(\sigma_i\bs;\bm,I)\\
       		&+|m_i|^{1-s_i}Q_{I_i,J_i(m,I)}^{\|\alpha_i\|^2}(s_i)Z^\ast(\sigma_i\bs;\bm,\sigma_i\bullet I),\nonumber
       		\end{align}
       	where for integers $i$ and $j$ the functions $P_{i,j}^t(s_k)$ and $Q_{i,j}^t(s_k)$ are defined by
	        \begin{align}
       	\label{eqn:PQ}
       	P_{i,j}^t(s_k)&=-(q(q^{-s_k}))^{1-(j+1-2i)_{n(\alpha_k)}}\frac{q-1}{1-q^{n(\alpha_k)+1}(q^{-s_k})^{n(\alpha_k)}}\\
       	Q_{i,j}^t(s_k)&=-\tau(\epsilon^{t(2i-j-1)})(q(q^{-s_k}))^{1-n(\alpha_k)}\frac{1-q^{n(\alpha_k)}(q^{-{s_k}})^{n(\alpha_k)}}{1-q^{n(\alpha_k)+1}(q^{-s_k})^{n(\alpha_k)}}.\nonumber
       	\end{align}
	We note that our functional equations \eqref{eqn:globfe} agree with those appearing in \cite{C}, which treats the untwisted case when $K=\F_q(T)$ and $\Phi=A_2$.

    Verifying \eqref{eqn:globfe} requires several steps. The crux of the argument uses Kubota's Dirichlet series, which is a rank-one Weyl group multiple Dirichlet series. Specifically, we fix $c_j$ for $j\neq i$ to obtain a new series $\mathcal{E}(s_i;\hat{\bc}_i;\bm,I_i)$ from $Z^\ast(\bs;\bm,I)$ and show that it satisfies the same functional equations as Kubota's Dirichlet series. Writing $Z^\ast(\bs;\bm,I)$ in terms of $\mathcal{E}(s_i;\hat{\bc}_i;\bm,I_i)$, we obtain \eqref{eqn:globfe}. We proceed as follows: first we define Kubota's Dirichlet series and state its functional equations. Then we define $\mathcal{E}(s_i;\hat{\bc}_i;\bm,I_i)$ and explain how its functional equations relate to that of Kubota's Dirichlet series. Finally, we use the former to obtain functional equations for $Z^\ast(\bs;\bm,I)$.

%
    \subsection{Kubota's Dirichlet series}
    	Kubota's Dirichlet series is a Weyl group multiple Dirichlet series associated to $\Phi=A_1$. This series may be defined for any global field $K$ and has applications to higher order reciprocity laws \cite{Kubota}. When $K=\F_q(T)$ we have
    	        \begin{equation}
    	        \label{eqn:kubotafq}
    	        D(s,m;\epsilon^t)=\sum_{c \in \mathcal{O}_{mon}} \frac{g_t(m,c)}{|c|^s}.
    	        \end{equation}
    	To state the functional equations, let $0\leq i\leq n-1$ and $m \in \mathcal{O}_{mon}$ and define      
	 \begin{equation}
        \label{eqn:kubotai}
        D(s,m;\epsilon^t,i)=\sum_{\substack{c \in \mathcal{O}_{mon}\\\deg c\equiv i \mod n}} \frac{g_t(m,c)}{|c|^s}.
        \end{equation}
      	It is shown in \cite[Proposition 2.1]{hoffstein} that the normalized series $D^\ast(s,m;\epsilon^k,i)=(1-q^{n-ns})^{-1}D(s,m;\epsilon^k,i)$ satisfies the following functional equation:
        \begin{align}
        \label{eqn:kubotaFE}
        D^\ast(s,m;\epsilon^t,i)&=|m|^{1-s}P_{i,\deg m}^t(s)D^\ast(2-s,m;\epsilon^t,i)\\
        &+ |m|^{1-s}Q_{i,\deg m}^t(s)D^\ast(2-s,m;\epsilon^t,\deg m+1-i),\nonumber
        \end{align}
        where $P_{i,j}^t$ and $Q_{i,j}^t$ are defined in \eqref{eqn:PQ}. Note that $P^t_{i,j}$ and $Q^t_{i,j}$ depend only on the value of ${2i-j} \mod n$.

	       	To compare our notation with that of \cite{wmnds}, we now follow \cite{C,patterson} and express the $D(s,m;\epsilon^t,i)$ as sums over equivalence classes of $K_\infty^\times/(K_\infty^\times)^n$. For $c,\eta \in K_\infty^{\times}$, we say that $c \sim \eta$ if and only if $c/\eta \in (K_\infty^\times)^n.$ For such an $\eta$, define
       	\begin{equation}
       	\label{eqn:kubotaD}
       	\mathcal{D}(s,m;\epsilon^t,\eta)=\sum_{c \sim \eta} \frac{g_t(m,c)}{|c|^s}.
       	\end{equation}
       	The following lemma shows $\mathcal{D}(s,m;\epsilon^t,\pi_\infty^{-i})=D(s,m;\epsilon^t,i)$:
       	           	\begin{lem}
       	           	\label{lem:equivalence}
       	           Let $c \in \mathcal{O}_{mon} \subset K_\infty^{\times}$. Then $c \sim \pi_\infty^{-i}$ if and only if $\deg c\equiv i \mod n$.
       	            \end{lem}
       	            \begin{proof}
       	            Suppose that $c/\pi_\infty^{-i}\in (K_\infty^\times)^n$. Then $c/\pi_\infty^{-i}=(f(\pi_\infty))^n$, where we assume
       	            \[
       	            f(\pi_\infty)=a_{-k}\pi_\infty^{-k}+a_{-k+1}\pi_\infty^{-k+1}+\cdots+a_0+\sum_{j=1}^\infty a_j\pi_\infty^j,
       	            \]
       	            has degree $-k$. Clearing the denominator, we have $c=\pi_\infty^{-i}(f(\pi_\infty))^n$ and $\deg c\equiv i \mod n$.
	            
       	            
       	            For the converse, suppose that $\deg c=k\equiv i \mod n$, and write
       	            \begin{align*}
       	            c
       	            &=\pi_\infty^{-i-kn}+a_1\pi_\infty^{-i-kn+1} +\cdots +a_{i+kn}. 
       	            \end{align*}
       	            Then
       	            \[
       	            c/\pi_\infty^{-i}=\pi_\infty^{-kn}+a_1\pi_\infty^{-kn+1}+\cdots+a_{i+kn}\pi_\infty^{i}=\pi_\infty^{-kn}(1+X),
       	            \]
       	            where $X=a_1\pi_\infty+\cdots+a_{i+kn}\pi_\infty^{i+kn} \in (\pi_\infty)$. 
       	            Define $f(u)=u^n-(1+X) \in K_\infty^\times[u]$. Then $(1+X) \in (K_\infty^\times)^n$ if and only if $f(u)=0$ has a solution in $K_\infty^\times$. Note that $u=1$ is a solution modulo $(\pi_\infty)$, and our assumption $q \equiv 1\mod {2n}$ implies $f'(u)$ is a unit. By Hensel's Lemma, there is a unique $K_\infty^\times$ solution to $f(u)=0$. 
       	            \end{proof}
       	  
	
	Let $0 \leq i\leq n-1$ and let $(\cdot,\cdot)_\infty$ be the Hilbert symbol at infinity.
	 Define
       	                \begin{equation*}
       	                \Psi_i(c)=\left\{\begin{array}{ccl} 1 &\mbox{if } c\sim\pi_\infty^{-i}&\mbox{and } c \in \mathcal{O}_{mon},\\
       	                (c,\epsilon)_\infty^{-t} &\mbox{if } c \sim \pi_\infty^{-i} &\mbox{and } c \mbox{ has leading coefficient } \epsilon,\\
       	                 0&\mbox{otherwise.}&\end{array}\right.
       	                \end{equation*}
Then we may also obtain $\mathcal{D}(s,m;\epsilon^t,\pi_\infty^{-i})=D(s,m;\epsilon^t,i)$ from $\mathcal{D}(s,m;\epsilon^t)$ by replacing the coefficients $g_t(m,c)$ with $g_t(m,c)\Psi_i(c)$ in \eqref{eqn:kubotaFE}. We claim that the effect of including $\Psi_i$ in the coefficients is to restrict the sum to the equivalence classes $[\pi_\infty^{-i}] \in K_\infty^\times/(K_\infty^\times)^n$. To see this, 
		let $\Omega=\F_q^\times (K_\infty^\times)^n$. Note that $\Omega$ is maximal isotropic for the Hilbert symbol in the sense that for any $\varepsilon_1,\varepsilon_2 \in \Omega$, we have $(\varepsilon_1,\varepsilon_2)_\infty=1$. Define $\mathcal{M}_t(\Omega)$ as the space of functions $\Psi:K_\infty^\times \to \C$ that satisfy 
       	       	       	                \begin{equation*}
       	       	       	                \Psi(\varepsilon c)=(c,\varepsilon)_\infty^{-t}\Psi(c)
       	       	       	                \end{equation*}
       	       	      for all $\varepsilon \in \Omega$.  Then the set $\{\Psi_i:0 \leq i \leq n-1\}$ forms a basis for 
		      $\mathcal{M}_t(\Omega)$. That is, any $\Psi \in \mathcal{M}_t(\Omega)$ is completely determined by its values on a set of representatives for $K_\infty^\times/\Omega$. Thus, $\dim \mathcal{M}(\Omega) = \dim K_\infty^\times/\Omega$. Further, Lemma \ref{lem:equivalence} shows that the representatives of $K_\infty^\times/\Omega$ are exactly $\pi_\infty^{-i}$, for $i\in\{0,\ldots, n-1\}$.

\subsection{The series $\mathcal{E}(s_i;{\bf a};\bm,I_i)$}
        To obtain the functional equations for $Z^\ast(\bs;\bm)$ we consider a new single-variable series. Fix ${\bf a}=(a_1,\ldots, \hat{a}_i,\ldots,a_r) \in (\mathcal{O}_{mon})^{r-1}$, where the hat means omit $a_i$. For $0 \leq j \leq n(\alpha_i)-1$, define
        \[
        \mathcal{E}(s_i,{\bf a};\bm,\pi_\infty^{-j}):=\sum_{\stackrel{c_i \in \mathcal{O}_{mon}}{c_i \sim \pi_\infty^{-j}}} \frac{H(a_1,\ldots,c_i,\ldots,a_r;\bm)}{|c_i|^{s_i}},
        \] 
        and the normalized series
        \[\mathcal{E}^\ast(s_i,{\bf a};\bm,\pi_\infty^{-j})=(1-q^{n(\alpha_i)(1-s_i)})^{-1}\mathcal{E}(s_i,{\bf a};\bm,\pi_\infty^{-j}).\]
       A special case of Theorem 5.8 of \cite{wmnds} shows that $\mathcal{E}^\ast$ satisfies functional equations of the same form as \eqref{eqn:kubotai}. We will sketch the proof.
   		\begin{thm}{\cite[Theorem 5.8]{wmnds}}
   		\label{thm:Efe}
   		Fix ${\bf a} \in (\mathcal{O}_{mon})^{r-1}$ and let $A=\prod_{j \neq i} a_j^{-c(j,i)}$. Set $\mathcal{E}(s_i,{\bf a};\bm,j)=\mathcal{E}(s_i,{\bf a};\bm,\pi_\infty^{-j})$. Then
   		\begin{align}
   		\label{eqn:Efe}
   		\mathcal{E}^\ast(s_i,{\bf a};\bm,j)&=|Am_i|^{1-s_i}P_{j,\deg m_i}^k\mathcal{E}^\ast(2-s_i,{\bf a};\bm,j)\\
   		&+|Am_i|^{1-s_i}Q_{j,\deg m_i}^k\mathcal{E}^\ast(2-s_i,{\bf a};\bm,\deg m_i+1-j).\nonumber
   		\end{align}
   		\end{thm}

   		   		\begin{proof}
   		Let $P$ be prime and ${\bf k}=(k_1,\ldots,k_r)$ an $r$-tuple of nonnegative integers. Recall that $i \in \{1,\ldots, r\}$ is fixed. Set $l_i=\ord_Pm_i$ and $n=n(\alpha_i)$. Writing $\beta=\sum k_j \alpha_j$, define a new tuple ${\bf k}'$ by setting $\sum k_j'\alpha_j=\sigma_i\bullet\beta$. Let
   		   		   		   		   		 \begin{align*}
   		   		   		   		   		 N^{(P;{\bf k})}(x;\bm,\alpha_i)&=\sum_{j \geq 0} H(P^{k_1},\ldots, P^{jn+(k_i)_{n}},\ldots,P^{k_r};\bm)x^{jn+(k_i)_{n}}.\\
   		   		   		   		   		 \intertext{Then define}
   		   		   		   		 f^{(P;{\bf k})}(x;\bm,\alpha_i)&=\frac{N^{(P;{\bf k})}(x;\bm,\alpha_i)}{1-p^{n-1}x^n}\\
   		   		   		   		 &-\delta_{k_i,k_i'}^mg(m_i^{-1}P^{l_i},P;\epsilon^{\|\alpha_i\|^2(k_i-k_i')})p^{(k_i-k_i'-1)_n}x^{(k_i-k_i')_n}\frac{N^{(P;{\bf k}')}(x;\bm,\alpha_i)}{1-p^{n-1}x^n}
   		   		   		   		 \end{align*}
   		   		   		   		 where $\delta_{i,j}^m$ is 0 if $i \equiv j \mod m$ and 1 otherwise. These functions satisfy the following functional equations \cite[Theorem 4.1]{wmnds}:
   		   		\begin{equation}
   		   		\label{eqn:ffe}
   		   		\frac{f^{(P;{\bf k})}(x;\bm,\alpha_i)}{f^{(P;{\bf k})}(1/(p^2x);\bm,\alpha_i)}=\left\{\begin{array}{cc} (px)^{l_i+1-(k_i'-k_i)_{n}} &\mbox{if } (k_i'-k_i)_n \neq 0,\\
   		   		(px)^{l_i+1-n} &\mbox{otherwise}.\end{array}\right.
   		   		\end{equation}
				
				We will write $\mathcal{E}^\ast(s_i,{\bf a};\bm,j)$ in terms of Kubota series $D(s_i,m;\epsilon^t,j)$ and $f^{(P;{\bf k})}(x;\bm,\alpha_i)$. 
   		To simplify notation, assume that $i=1$. Let $P_1,\ldots,P_v$ be the prime divisors of $a_2\cdots a_r m_1\cdots m_r$, with $p_j=|P_j|$. Let $S=\{P_1,\ldots,P_v\}$. Write $a_j=P^{\beta_{j1}}\cdots P^{\beta_{jv}}$ for $j=2,\ldots, r$ and $Am_1=P^{l_1}\cdots P^{l_r}$. One can show that
   		\begin{align}
   		\label{eqn:Erewrite}
   		\mathcal{E}(s_1,{\bf a};\bm,j)&=\xi \sum_{k_1,\ldots,k_v=0}^{m-1}\mathcal{D}(s_1,P^{(l_1-2k_1)_n}\cdots P_v^{(l_v-2k_v)_n},\epsilon^{\|\alpha_1\|},P_1^{-k_1}\cdots P_v^{-k_v}\pi_\infty^{-j})\\\nonumber
   		&\times C(k_1,\ldots, k_r)\prod_{i=1}^v f^{(P_i;k_i,b_{21},\ldots,b_{ri})}(p_i^{-s_1};\bm,\alpha_1),
   		\end{align}
   		where $\xi$ and $C(k_1,\ldots,k_r)$ are products of residue symbols. In particular, for
   		   		\(\eta' \sim P^{2k_1-l_1-1}\) and $K_j=l_j-2k_j$, we have \cite[Lemma 5.9]{wmnds}
   		   		\begin{align}
   		   		\label{eqn:Cfe}
   		   		&\frac{C(k_1,\ldots,k_r){\Psi}_{I_1}^{(P_1^{k_1}\cdots P_v^{k_v})}}{C(l_1-k_1+1,\ldots, l_r-k_r+1)\hat{ \Psi}_{{\eta'}}^{(P_1^{(l_1+1-k_1)_n}\cdots P_v^{(l_v+1-k_v)_n})}}=\leg{m_1P_2^{K_2}\cdots P_v^{K_v}}{P_1^{2k_1-l_1-1}}^{-\|\alpha_1\|^2},
   		   		\end{align}
   		 where we define $\Psi^{(a)}(c)=\Psi(ac)$ and $\hat{\Psi}_\eta(c)=(\eta,c)_\infty^t\Psi(\eta c)$.
It is clear that $\hat{\Psi}_\eta\in\mathcal{M}_t(\Omega)$ and depends only on the class of $\eta \in K_\infty^\times/(K_\infty^\times)^n$.
   		   				   	
   		 The proof of \eqref{eqn:Erewrite} is a lengthy, but straightforward, computation with residue symbols. The idea is to use twisted multiplicativity to rewrite the coefficients $H(c_1,a_2,\ldots,a_r)$ in terms of Gauss sums and prime power coefficients. This follows from considering $c_1=cc'$, where we assume that $(c,a_2\cdots a_r m_1\cdots m_r)=1$. Summing all relevant $c$, up to a product of residue symbols we can write $\mathcal{E}(s_1,{\bf a};\bm,j)$ as the sum of the product of Kubota series of the form $\mathcal{D}_S(s_1,m;\epsilon^{\|\alpha_1\|^2},\eta)$ and polynomials $N^{(P;{\bf k})}(p_i^{-s_1};\bm,\alpha_1)$. Here $\mathcal{D}_S(s,m;\epsilon^t,\eta)$ is a generalization of \eqref{eqn:kubotaD}; for a finite set of primes $S$, it is defined exactly the same as \eqref{eqn:kubotaD} except that we restrict the sum to those $c$ relatively prime to the elements of $S$.
   		 
   		 To obtain $\mathcal{D}(s_1,m;\epsilon^{\|\alpha_1\|^2},\eta)$ from  $\mathcal{D}_S(s_1,m;\epsilon^{\|\alpha_1\|^2},\eta)$, we use the following result of \cite{patterson} (see also \cite[Lemma 5.4]{wmnds}) to ``remove'' primes from $S$ one at a time.
   		 	\begin{align}
   		 	\label{eqn:kubotaS}
   		   	D_{S\cup\{P\}}(s,mP^i;\epsilon^t,\pi_\infty^{-j})&=\frac{D_S(s,mP^i;\epsilon^t,\pi_\infty^{-j})}{1-p^{n-1-ns}}\\
   		    &-\frac{g(mP^i,P^{i+1};\epsilon^t)}{p^{(i+1)s}}\frac{D_S(s,mP^{n-i-2};\epsilon^t,\pi_\infty^{-i-j-1})}{1-p^{n-1-ns}},\nonumber
   		    \end{align}
   		 Making a change of variables and applying \eqref{eqn:Cfe}, we put the two terms on the left-hand side of \eqref{eqn:kubotaS} together to obtain \eqref{eqn:Erewrite} after $v$ iterations.
   		 
   		It remains to see that \eqref{eqn:Erewrite} satisfies \eqref{eqn:Efe}. For this, we first apply the functional equations of $D^\ast$. Assume ${\bf k}=(k_1,\ldots, k_r)$, where each $k_j \in \{0,\ldots, n(\alpha_1)-1\}$, and let 
   		\begin{align*}
   		E&=E({\bf k})=P^{(l_1-2k_1)_n}\cdots P_v^{(l_v-2k_v)_m}\\
   		F&=F({\bf k})=P^{-k_1}\cdots P_v^{-k_v}.
   		\end{align*} 
   		Set $e=\deg E$ and $f=\deg F$. It follows from \eqref{eqn:kubotaFE} that
   		\begin{align*}
   		D^\ast(s_1,E;\epsilon^{\|\alpha_1\|^2},i-f)&=|E|^{1-s_1}P_{i-f,e}^{\|\alpha_1\|^2}(s_1)D^\ast(2-s_1,E;\epsilon^{\|\alpha_1\|^2},i-f)\\
   		&+Q_{i-f,e}^{\|\alpha_1\|^2}(s_1)D^\ast(2-s_1,E;\epsilon^{\|\alpha_1\|^2},e+1-(i-f)).
   		\end{align*}
   		Recall that $P_{i,j}^{\|\alpha_1\|^2}$ and $Q_{i,j}^{\|\alpha_1\|^2}$ depend only on the value of $2i-j$ modulo $n(\alpha_1)$. We have
   		\begin{align*}
   		2i-2f-e&=2i-\sum_{j=1}^v\deg P_j(-2k_j-(l_j-2k_j)_n)\\
   		&\equiv 2i -n+\sum_{j=1}^v \deg P_j (l_j)_n\\
   		&\equiv 2i-\deg Am_1 \mod n.
   		\end{align*}
   		Therefore $P_{i-f,e}(s_1)=P_{i,\deg Am_1}(s_1)$ and $Q_{i-f,e}(s_1)=Q_{i,\deg Am_1}(s_1)$ do not depend on ${\bf k}$.
   		The result now follows from \cite{wmnds} using \eqref{eqn:ffe}.
   		\end{proof}
	
   		\subsection{Functional Equations for $Z^\ast(\bs;\bm)$}
   		We now derive the functional equations for $Z^\ast(\bs;\bm)$. Fix a simple reflection $\sigma_i \in W$ and let $n=n(\alpha_i)$. We have the following decomposition:
   		\begin{align}
   		\label{eqn:z1st}
   		Z^\ast(\bs;\bm,I)=&\Xi(\bs)\sum_{\substack{\bc=(\mathcal{O}_{mon})^r\\c_j \equiv I_j\mod{n(\alpha_j)}}} \frac{H(c_1,\ldots,c_i,\ldots, c_r;\bm)}{|c_1|^{s_1}\cdots |c_r|^{s_r}}\\
   		=&\frac{\Xi(\bs)}{\zeta(ns_i-n+1)}\sum_{\substack{\bc=(\mathcal{O}_{mon})^{r-1}\\c_j \equiv I_j\mod{n(\alpha_j)}}}\frac{1}{\prod_{j=1, j\neq i}^r |c_j|^{s_j}}\mathcal{E}^\ast(s_i,\hat{\bc}_i;\bm,I_i).\nonumber
   		\end{align}
   		Let $C=\prod_{j\neq i}c_j^{c(j,i)}$. Applying \eqref{eqn:Efe} to $\mathcal{E}^\ast(s_i,\hat{\bc}_i;\bm,I_i,i)$, we see \eqref{eqn:z1st} equals
   		\begin{align}
   		\label{eqn:z2nd}
   		&\frac{\Xi(\bs)}{\zeta(ns_i-n+1)}\sum_{\stackrel{\bc=(\mathcal{O}_{mon})^{r-1}}{c_j \equiv I_j\mod n(\alpha_j)}}\frac{1}{\prod_{j=1, j\neq i}^r |c_j|^{s_j}}|Cm_i|^{1-s_i}\\
   		&\times\bigg(P_{I_i,\deg{Cm_i}}^{\|\alpha_i\|^2}(s_i)\mathcal{E}^\ast(2-s_i,\hat{\bc}_i;\bm,I_i)+Q_{I_i,\deg{Cm_i}}^{\|\alpha_i\|^2}(s_i)\mathcal{E}^\ast(2-s_i,\hat{\bc}_i;\bm,\deg{Cm_i}+1-I_i)\bigg).\nonumber
   		\end{align}
   		Substituting $\deg{Cm_i}=\deg m_i-\sum_{j \neq i}c(j,i)\deg c_j$ into \eqref{eqn:z2nd}, we have
   		\begin{align*}
   		&\frac{\Xi(\bs)}{\zeta(ns_i-n+1)}|m_i|^{1-s_i}\sum_{\stackrel{\bc=(\mathcal{O}_{mon})^{r-1}}{c_j \equiv I_j\mod n(\alpha_j)}}\frac{1}{\prod_{j=1, j\neq i}^r |c_j|^{s_j-c(j,i)(s_i-1)}}\\
   		&\times\bigg(P_{I_i,\deg m_i-\sum_{j \neq i}c(j,i)I_j}^{\|\alpha_i\|^2}(s_i)\mathcal{E}^\ast(2-s_i,\hat{\bc}_i;\bm,I_i)\\
   		&+Q_{I_i,\deg m_i-\sum_{j \neq i}c(j,i)I_j}^{\|\alpha_i\|^2}(s_i)\mathcal{E}^\ast(2-s_i,\hat{\bc}_i;\bm,\deg m_i-\sum_{j \neq i}c(j,i)I_j+1-I_i)\bigg).
   		\end{align*}
Recall that $\sigma_i$ permutes the positive roots of $\Phi$ other than $\alpha_i$. It follows that  \[\frac{\Xi(\bs)}{\zeta(ns_i-n+1)}=\frac{\Xi(\sigma_i\bs)}{\zeta(n(2-s_i)-n+1)}.\] Thus we have achieved the desired result.
   		
   		It will be convenient to express \eqref{eqn:globfe} in a slightly different way. Summing $Z^\ast(\bs;\bm,I)$ over all $I$, we have
   		\begin{equation}
		\label{eqn:globfe2}
		Z^\ast(\bs;\bm)=|m_i|^{1-s_i}\sum_{I}\bigg(P_{I_i,\deg m_i-\sum_{j \neq i}c(j,i)I_j}^{\|\alpha_i\|^2}+Q_{(\sigma_i\bullet I)_i,\deg m_i-\sum_{j\neq i} c(j,i)I_j}^{\|\alpha_i\|^2}\bigg)Z^\ast(\bs;\bm;I).
   		\end{equation}
		We remark on the similarity between the $p$-part functional equations \eqref{eqn:ppfe} and the global functional equations \eqref{eqn:globfe2}.
		
		
\section{The Support of $Z^\ast(\bs;\bm)$}
\label{sec:global}
In this section we prove the rationality of and describe the support of $Z^\ast(\bs;\bm)$. Let $X_i=q^{-s_i}$ and $\X=(X_1,\ldots,X_r)$. Under this identification, we put $\mathcal{Z}^\ast(\X;\bm)=Z^\ast(\bs;\bm)$.

\subsection{Rationality of $Z^*(\X;\bm)$}
The key to understanding the relationship between the global series and its $p$-parts is to note that both $F(\x;\ell)$ and $\mathcal{Z}^\ast(\X;\bm)$ are rational functions that, up to a change of variables, satisfy the same functional equations. Thus, we first show that $\mathcal{Z}^\ast(\X;\bm)$ is indeed a rational function.  

\begin{thm}
\label{thm:rat}Write $\sum_{\alpha \in \Phi^+} \alpha=\sum_{i=1}^r \rho_i \alpha_i$ so that $\rho_i$ is the coefficient of $\alpha_i$ in the sum of positive roots of $\Phi$. For a fixed choice of reduced expression $w_0=\sigma_{i_t}\sigma_{i_{t-1}}\cdots\sigma_{i_1}$, let $\beta_{j}$ for $j=1,\ldots, t$ be the positive root corresponding to $\sigma_{i_j}$. That is, $\beta_{j}= \sigma_{i_t}\sigma_{i_{t-1}}\cdots\sigma_{i_{j+1}}(\alpha_{i_j})$ cf.~\cite[VI: Corollary 2 to Proposition 17]{bourbaki}. For fixed $l$, write $\sum_{i_j=l} \beta_l=\sum_{k=1}^r c_{l_k}\alpha_k$.

Finally, set $D(\X)=\prod_{\alpha>0} (1-q^{n(\alpha)d(\alpha)+1}\X^{n(\alpha)\alpha})$. Then $\mathcal{N}(\X;\bm)=D(\X)\mathcal{Z}^\ast(\X;\bm)$ is polynomial in $X_1,\ldots, X_r$ of degree at most $\rho_i+\sum_{k=1}^r c_{k_i}\deg m_k$ in each $X_i$.
\end{thm}

Our proof requires the following root-theoretic statement.

\begin{thm}
\label{thm:posroots}
Fix a reduced expression $w_0=\sigma_{i_t}\sigma_{i_{t-1}}\cdots\sigma_{i_1}$, where necessarily $t=\#\Phi^+$. Let $\beta_{j}$ for $j=1,\ldots, t$ be the positive roots corresponding to $\sigma_{i_j}$. That is, $\beta_{j}= \sigma_{i_t}\sigma_{i_{t-1}}\cdots\sigma_{i_{j+1}}(\alpha_{i_j})$ cf.~\cite[VI: Corollary 2 to Proposition 17]{bourbaki}. Then 
\begin{equation}
\label{eqn:posrooteq}
\sum_{i_j=k} \beta_k=\varpi_k-w_0\varpi_k.
\end{equation}
\end{thm}

\begin{rem*} If $w_0=-1$, then the right-hand side of \eqref{eqn:posrooteq} becomes $2\varpi_k$. If we sum \eqref{eqn:posrooteq} over all $k$, we get the two standard expressions for $2\rho$:  $\sum_{\alpha \in \Phi^+} \alpha=2\sum_{i=1}^r \varpi_i$. Thus Theorem \ref{thm:posroots} can be seen as a refinement of the two different ways of computing $2\rho$. \end{rem*}

\begin{proof}[Proof of Theorem \ref{thm:posroots}]
Our proof of Theorem \ref{thm:posroots} for the classical cases utilizes a specific {\em good} choice of reduced decomposition for $w_0$ due to Littelmann \cite{littelmann}. The exceptional cases $G_2$, $F_4$, $E_6$, $E_7$, and $E_8$ were checked via computer computation with LiE. In what follows, $(\epsilon_1,\ldots, \epsilon_r)$ is the canonical basis on $\R^r$ (see \cite[VI, 4.]{bourbaki}). Below we use nonstandard terminology and refer to the positive roots $\beta_k$ corresponding to the simple reflections $\sigma_k$ in the reduced expression  for $w_0$ as {\em $k$-roots}. 

Let $\Phi=A_r$. The simple roots are $\alpha_1=\epsilon_1-\epsilon_2,\ldots,\alpha_r=\epsilon_r-\epsilon_{r+1}$. For $A_1$, $w_0=-1$ and \eqref{eqn:posrooteq} is easily verified. For $r \geq 2$, $w_0$ is the automorphism of $\Phi$ that transforms $\alpha_i$ to $-\alpha_{r+1-i}$. In this case the fundamental weights are $\varpi_i=\epsilon_1+\cdots+\epsilon_i-\frac{i}{r+1}\sum_{j=1}^{r+1}\epsilon_j$ (see for example, \cite[Plate I]{bourbaki}). Following \cite{littelmann}, the good decomposition of the long word is 
\[w_0=\sigma_1(\sigma_2\sigma_1)(\sigma_3\sigma_2\sigma_1)(\ldots)(\sigma_r\sigma_{r-1}\cdots \sigma_2\sigma_1),\]
and the corresponding enumeration of the positive roots is 
\begin{align*}
\beta_1&=\epsilon_1-\epsilon_2;\beta_2=\epsilon_1-\epsilon_3;\beta_3=\epsilon_2-\epsilon_3;\beta_4=\epsilon_1-\epsilon_4,\ldots,\beta_6=\epsilon_3-\epsilon_4;\ldots;\\
\beta_{R-r+1}&=\epsilon_1-\epsilon_{r+1};\ldots;\beta_R=\epsilon_r-\epsilon_{r+1},\end{align*}
where $R=\frac{1}{2}r(r+1)$. One readily checks that the $k$-roots are of the form $\epsilon_i-\epsilon_j$ such that $1 \leq i<j\leq r+1$ and $j-i=k$. It follows that $\sum_{i_j=k}\beta_j=\sum_{i_j=r-k+1} \beta_j$. Therefore, we assume without loss of generality that $k\leq r/2$. It follows from our definition of $\varpi_k$ that $\varpi_k-w_0\varpi_k=\varpi_k+\varpi_{r-k+1}=(\epsilon_1+\cdots+\epsilon_{r-k+1})-(\epsilon_{k+1}+\cdots+\epsilon_{r+1})$. Clearly this agrees with the sum of the $k$-roots. 
\medskip

Let $\Phi=B_r$ or $\Phi=C_r$. As in \cite{littelmann}, we label the simple roots $\alpha_r=\epsilon_1-\epsilon_2,\ldots,\alpha_2=\epsilon_{r-1}-\epsilon_r$, $\alpha_1=\epsilon_r$ for $B_r$ and respectively $\alpha_1=2\epsilon_r$ for $C_r$. (This is different from the usual enumeration in \cite{bourbaki}!)  With this labeling, we have $\varpi_k=\sum_{i=1}^{r-k+1} \epsilon_i$ for $k=2,\ldots r$ and $\varpi_1=\frac{1}{2}(\epsilon_1+\epsilon_2+\cdots+\epsilon_r)$ for $B_r$ and respectively $\varpi_1=\epsilon_1+\epsilon_2+\cdots+\epsilon_r$ for $C_r$. Note that in both cases $w_0=-1$. 

The good decomposition of the long word is
\[w_0=\sigma_1(\sigma_2\sigma_1\sigma_2)(\ldots)(\sigma_{r-1}\ldots\sigma_1\ldots\sigma_{r-1})(\sigma_r\sigma_{r-1}\ldots\sigma_1\ldots\sigma_{r-1}\sigma_r),\]
and the corresponding enumeration of the positive roots for $B_r$ is
\begin{align*}
\beta_1&=\epsilon_r;\beta_2=\epsilon_{r-1}+\epsilon_r,\beta_3=\epsilon_{r-1},\beta_4=\epsilon_{r-1}-\epsilon_r;\ldots;\\
\beta_{R-2r+2}&=\epsilon_1+\epsilon_2,\ldots,\beta_{R-r+1} \epsilon_1;\ldots \beta_r=\epsilon_1-\epsilon_2;
\end{align*}
where $R=r^2$. The enumeration for $C_r$ is 
\begin{align*}
\beta_1&=2\epsilon_r;\beta_2=\epsilon_{r-1}+\epsilon_r,\beta_3=2\epsilon_{r-1},\beta_4=\epsilon_{r-1}-\epsilon_r;\ldots;\\
\beta_{R-2r+2}&=\epsilon_1+\epsilon_2,\ldots,\beta_{R-r+1} 2\epsilon_1;\ldots \beta_r=\epsilon_1-\epsilon_2.
\end{align*}
One easily checks \eqref{eqn:posrooteq} holds in all cases. That is, the $k$-roots sum to $2\varpi_k$. 
\medskip

Let $\Phi=D_r$. As in \cite{littelmann}, we label the simple roots $\alpha_r=\epsilon_1-\epsilon_2,\ldots,\alpha_2=\epsilon_{r-1}-\epsilon_r$ and $\alpha_1=\epsilon_{r-1}+\epsilon_r$. (This is different from the usual enumeration in \cite{bourbaki}!) With this labeling, we have $\varpi_1=\frac{1}{2}(\epsilon_1+\epsilon_2+\cdots +\epsilon_{r-2}+\epsilon_{r-1}+\epsilon_r)$; $\varpi_2=\frac{1}{2}(\epsilon_1+\epsilon_2+\cdots+\epsilon_{r-2}+\epsilon_{r-1}-\epsilon_r)$; and $\varpi_i=\epsilon_1+\epsilon_2+\cdots +\epsilon_{r-i+1}$ for $3 \leq i \leq r$. In this case we have $w_0=-1$ if $r$ is even and $w_0=-\varepsilon$ if $r$ is odd, where $\varepsilon$ is the automorphism of $\Phi$ that interchanges $\alpha_1$ and $\alpha_2$ and leaves the other $\alpha_i$ fixed.

The good decomposition of the long word is 
\[\sigma_1\sigma_2(\sigma_3\sigma_1\sigma_2\sigma_3)(\sigma_4\sigma_3\sigma_1\sigma_2\sigma_3\sigma_4)(\ldots)(\sigma_r\ldots\sigma_4\sigma_3\sigma_1\sigma_2\sigma_3\sigma_4\ldots\sigma_r),\]
and the corresponding enumeration of the positive roots is
\begin{align*}
\beta_1&=\epsilon_{r-1}+\epsilon_r,\beta_2=\epsilon_{r-1}-\epsilon_r;\beta_3=\epsilon_{r-2}+\epsilon_{r-1},\beta_4=\epsilon_{r-2}-\epsilon_r\\
\beta_5&=\epsilon_{r-2}+\epsilon_r,\beta_6=\epsilon_{r-2}-\epsilon_{r-1};\ldots\\
\beta_{R-2r+3}&=\epsilon_1+\epsilon_2,\ldots,\beta_{R-r}=\epsilon_r+\epsilon_{r-1},\beta_{R-r+1}=\epsilon_1-\epsilon_r\\
\beta_{R-r+2}&=\epsilon_1+\epsilon_r,\beta_{R-r+3}=\epsilon_1-\epsilon_{r-1},\ldots,\beta_R=\epsilon_1-\epsilon_2;
\end{align*}
where $R=r^2-r$. With this labeling, Equation \eqref{eqn:posrooteq} is easily verified. 

\end{proof}

\begin{proof}[Proof of Theorem \ref{thm:rat}]
Our arguments generalize that of Theorem 5.1 of Fisher--Friedberg \cite{FF} and Proposition 4.1 of Chinta \cite{C}, which treat the case $\Phi=A_2$ and $\bm=(1,1)$.  First note that Theorem 6.1 of \cite{wmnds} shows $\mathcal{Z}^\ast(\X;\bm)$ is meromorphic and its set of polar hyperplanes is contained in the $W$-translates of the hyperplanes $s_i=1\pm 1/\gcd(n,\|\alpha_i\|^2)$. Thus, by arguments similar to the proof of Theorem 2 in \cite{wmds1}, the function $\mathcal{N}(\X;\bm)$ is entire. To prove that it is polynomial, we bound the degree. We consider two cases, dependent on the action of $w_0$ on $\Phi$. 
\begin{itemize}
\item Case 1: $A_1$, $\Phi=B_r$, $C_r$, $D_r$ with $r$ even, $E_7$, $E_8$, $F_4$, or $G_2$.
\item Case 2: $\Phi=A_r$ with $r \geq 2$, $D_r$ with $r$ odd, or $E_6$.
\end{itemize} 




We first address Case 1. In this case, we have $w_0=-1$. For now, assume that $\gcd(n,\|\alpha\|)=1$ for all positive roots $\alpha$ and $\bm=(1,\ldots, 1)$. Let $\vec{\mathcal{Z}}^*(\X;\bm)$ be the column vector consisting of all $\mathcal{Z}^*(\X;\bm,I)$ with $I=(I_1,\ldots,I_r)$ satisfying $I_j \in \{0,\ldots, n-1\}$. In matrix notation, the functional equation \eqref{eqn:globfe} can be expressed as
\begin{equation}
\label{eqn:matrxife}
\vec{\mathcal{Z}}^*(\X;\bm)=A_{\sigma_i}(X_i)\vec{\mathcal{Z}}^*(\sigma_i\X;\bm)
\end{equation}
where $A_{\sigma_i}(X_i)$ is an $n^r \times n^r$ matrix whose coefficients are the functions $P_{i,j}^t(X_i)$ and $Q_{i,j}^t(X_i)$. We note that by definition, each $A_{\sigma_i}(X_i)\ll X_i^{1-n}$. Let $t=\#\Phi^+$ and fix a reduced expression $w_0=\sigma_{i_1}\sigma_{i_2}\cdots\sigma_{i_t}$. Let $A_{w_0}(\X)$ be the matrix product corresponding to repeated application of the functional equations for this this choice. Indeed we can write $A_{w_0}(\X)$ explicitly as $\prod_{j=1}^{t} A_{\sigma_{i_j}}(\sigma_{i_t}\sigma_{i_{t-1}}\cdots\sigma_{i_{j+1}}\x^{\alpha_{i_j}})$. Note that by Bourbaki VI Corollary 2 to Proposition 17, the $\sigma_{i_t}\sigma_{i_{t-1}}\cdots\sigma_{i_{j+1}}\x^{\alpha_{i_j}}$ correspond to the distinct positive roots of $\Phi$, with some extraneous $q$ powers.  Thus, we can express the functional equation \eqref{eqn:globfe} in matrix form as
\begin{equation}
\label{eqn:vectorfe}\vec{\mathcal{Z}}^*(\X;\bm)=A_{w_0}(\X)\vec{\mathcal{Z}}^*(\frac{1}{q^2X_1},\ldots, \frac{1}{q^2X_r};\bm).
\end{equation}
Multiplying both sides by $D(\X)D(\frac{1}{q^2X_1},\ldots,\frac{1}{q^2X_r})$ yields
\begin{equation}
\label{eqn:rationalized}
D(\frac{1}{q^2X_1},\ldots,\frac{1}{q^2X_r})\vec{\mathcal{N}}(\X;\bm)=D(\X)A_{w_0}(\X)\vec{\mathcal{N}}(\frac{1}{q^2X_1},\ldots,\frac{1}{q^2X_r};\bm)
\end{equation}
where $\vec{\mathcal{N}}(\X;\bm)$ is the vector with components $D(\X)\vec{\mathcal{Z}}^*(\X;\bm,I)$. It suffices now to show that each entry of $\vec{\mathcal{N}}(\X;\bm)$ is of degree at most $\rho_i$ in each $X_i$. Let $X_i \to \infty$ in \eqref{eqn:rationalized}. The terms $D(\frac{1}{q^2X_1},\ldots,\frac{1}{q^2X_r})$ and $\vec{\mathcal{N}}(\frac{1}{q^2X_1},\ldots,\frac{1}{q^2X_r};\bm)$ remain bounded, while
\[D(\X)=O(|\prod_{i=1}^r X_i^{\rho_i n}|)\mbox{ and }A_{w_0}(\X)=O(|\prod_{i=1}^r X_i^{\rho_i-\rho_in}|).\]
Therefore the right hand side is $O(\prod_{i=1}^r X_i^{\rho_i})$.
Thus, $\mathcal{N}(\X;\bm)$ is a polynomial of degree at most $\rho_i$ in each $X_i$.

If $\gcd(n,\|\alpha\|)> 1$ for some $\alpha \in \Phi$, then our vector $\vec{\mathcal{Z}}^*(\X,\bm)$ has length $\prod_{i=1}^r n(\alpha_i)$. Still, we derive \eqref{eqn:rationalized} is the same way, and we readily check that as $X_i \to \infty$, $D(\X)=O(|\prod_{i=1}^rX_i^{{\rho_i}n(\alpha_i)}|$ and $A_{w_0}(\X)=O(|\prod_{i=1}^r X_i^{\rho_i-\rho_in(\alpha_i)}|)$. 

For general $\bm$, we have $A_{\sigma_i}(X_i)\ll X_i^{1-n(\alpha_i)+\deg m_i}$.  In the notation of theorem, we write $A_{w_0}(\X)=\prod_{j=1}^{t} A_{\sigma_{i_j}}(q^*\X^{\theta_{j}})$, where $q^*$ represents some extraneous $q$ powers. 
It follows that \[A_{w_0}(\X)=O(|\prod_{i=1}^r X_i^{\rho_i-\rho_in(\alpha_i)+\sum_{k=1}^rc_{k_i}\deg m_k}|).\] 
The same strategy as above shows that $\mathcal{N}(\X;\bm)$ is polynomial in $X_1,\ldots, X_r$ of degree at most $\rho_i+\sum_{k=1}^r c_{k_i}\deg m_k$ in each $X_i$. Note that this expression is independent of our choice of reduced expression for $w_0$ by Theorem \ref{thm:posroots}.


In Case 2 it is not true that $w_0=-1$. However, here we have additional automorphisms of $\Phi$ (see \cite[VI]{bourbaki}) that yield extra functional equations (in addition to that of \eqref{eqn:globfe}), which correspond to relabeling the variables $X_1,\ldots,X_r$. In each case, let $B$ represent the matrix corresponding to the functional equation as follows: for $\Phi=A_r$ and $r \geq 2$, the functional equation that interchanges $s_i$ and $s_{r-i+1}$; for $\Phi=D_r$ and $r$ odd, the functional equation that interchanges $s_{r-1}$ and $s_r$; and for $E_6$, the functional equation that exchanges $s_1,s_2,s_3,s_4,s_5,s_6$ with $s_6,s_2,s_5,s_4,s_3,s_1$, respectively. Now for each case separately, replace \eqref{eqn:vectorfe} with
\[\vec{\mathcal{Z}}^*(\X;\bm)=BA_{w_0}(\X)\vec{\mathcal{Z}}^*(\frac{1}{q^2X_1},\ldots, \frac{1}{q^2X_r};\bm).\]
The remainder of the argument for Case 1 now follows in the exact same way.
\end{proof}

\subsection{The Variable Change}
\label{sec:untwisted}
In what follows, we abuse terminology and refer to both $F(\x;\theta)=F(\x;\ell)$ and $N(\x;\theta)=N(\x;\ell)$ as $p$-parts. (Previously, we referred to $F(\x;\ell)$ as {\em modified} $p$-parts.) Recall that we say $Z(\bs;\bm)$ is untwisted when $\bm=(1,\ldots,1)$. Similarly, we say that the $p$-part $F(\x;\ell)$ is untwisted when $\ell=(0,\ldots,0)$. It was noticed in\cite{C}, for $\Phi=A_2$, and \cite{chintamohler}, for $\Phi=A_r$ and $n\gg r$, that a variable change transforms $F(\x;{\bf 0})$ to $\mathcal{Z}^\ast(\X;{\bf 1})$. We now generalize this observation to all $\Phi$ and $n$.

\begin{prop}
\label{prop:localglobal}
Let $\tilde{F}(\X;0,\ldots,0)$ denote $F(\x;0,\ldots,0)$ after the variable change
\begin{equation}
\label{eqn:change}
\left\{\begin{array}{lll}s_i &\mapsto& 2-s_i\\ p=|P| &\mapsto& 1/q\\g^\ast_k(P)&\mapsto&\tau(\epsilon^k) \end{array}\right.,
\end{equation}
where $\tau(\epsilon^k)$ and $g_k^*(P)$are the Gauss sums defined by \eqref{eqn:gausssum} and \eqref{eqn:globgauss}, respectively. Then
\(
\mathcal{Z}^\ast(\X;1,\ldots,1)=\tilde{F}(\X;0,\ldots,0).
\)
\end{prop}
\begin{proof}
Let $\beta=\sum_{j=1}^r \beta_j \alpha_j$ and let $(\sigma_i\bullet\beta)_j$ denote the $\alpha_j$ coefficient of $\sigma_i \bullet \beta$. One computes
\[
(\sigma_i\bullet\beta)_j=\left\{\begin{array}{ll} \beta_j&\mbox{ if } i \neq j,\\
1-\beta_i-\sum_{j \neq i} \beta_jc(j,i) &\mbox{ if } i= j.\end{array}\right.\] 
It follows that $\delta_i(\beta)=1-\sum_{j\neq i} c(j,i)\beta_j-2\beta_i$. Put $I=\beta_i$ and $J=-\sum_{j\neq i} c(j,i)\beta_j$. Then the transformation \eqref{eqn:change} takes $\mathcal{P}_{\beta,{\bf 0},i}(x_i)$ and $\mathcal{Q}_{\beta,{\bf 0},i}(x_i)$ of \eqref{eqn:locPQ} to $P_{I,J}^{\|\alpha_i\|^2}(s_i)$ and $Q_{I,J}^{\|\alpha_i\|^2}(s_i)$ of \eqref{eqn:PQ}, under the assumption $\bm=(1,\ldots, 1)$. Thus the functional equations \eqref{eqn:ppfe} of $F(\x;0,\ldots,0)$ and \eqref{eqn:globfe2} of $\mathcal{Z}^\ast(\X;1,\ldots,1)$ coincide up to this change of variables.

Next, notice that $\mathcal{Z}^\ast(\X;1,\ldots,1)$ and $\tilde{F}(\X;0,\ldots,0)$ have the same polar behavior.  This follows from Theorem \ref{thm:rat} and the observation that
substituting $\zeta_\mathcal{O}(s)=(1-q^{1-s})^{-1}$ we can rewrite 
\begin{align*}
\Xi(\X)&=\prod_{\alpha=\sum k_i\alpha_i>0} \zeta_\mathcal{O}(1+n(\alpha)\sum_{i=1}^r k_i(s_i-1))\\
&=\prod_{\alpha=\sum k_i\alpha_i>0} (1-q^{1-(1-n(\alpha)d(\alpha)+n(\alpha)\sum_{i=1}^rk_is_i)})^{-1}\\
&=\prod_{\alpha>0}(1-q^{n(\alpha)d(\alpha)}\X^{n(\alpha)\alpha})^{-1}.
\end{align*}                              
Then \eqref{eqn:change} takes $p^nx_i^n$ to $q^nX_i^n$ and hence $\Delta(\x)^{-1}$ to $\Xi(\X)$.

Finally, both $\mathcal{Z}^\ast(\X;1,\ldots,1)$ and $\tilde{F}(\X;0,\ldots,0)$ have constant term equal to one. Applying Theorem \ref{thm:unstab},
both functions are uniquely determined by their functional equations; thus, they must be equal.
\end{proof}

\subsection{The Support of $Z^*(\bs;\bm)$}
\label{sec:twisted}
We can now describe the support of $Z^*(\bs;\bm)$ for any $\bm$.

\begin{thm}
\label{thm:localglobal}
Fix $\bm \in \mathcal{O}^r$ and put $\ell=(\deg m_1,\ldots,\deg m_r)$. Define $\theta$ according to \eqref{eqn:theta}, and let $\tilde{F}(\X;\theta)$ denote $F(\x;\theta)$ after applying \eqref{eqn:change}. Let $\Theta$ be the set of dominant weights in the representation of highest weight $\theta$, and let $\Theta^+\subset\Theta$ be the subset of regular dominant weights. Then\begin{equation*}
\mathcal{Z}^\ast(\X;\bm)=\sum_{\xi \in \Theta^+} M_{\theta-\xi}\tilde{F}(\X;\xi)\X^{\theta-\xi},
\end{equation*}
where for $\lambda=\sum_{i=1}^r \lambda_i \alpha_i$, the coefficients $M_\lambda$ are the character sums 
\begin{equation*}
\label{eqn:mult}
M_\lambda=\sum_{\stackrel{{\bf c} \in (\mathcal{O}_{mon})^r}{\deg c_i=\lambda_i}} H({\bf c};{\bf m}).
\end{equation*} 
\end{thm}
\begin{rem*}
The $M_\lambda$ are the $\X^\lambda$ coefficients of $Z(\bs;\bm)$ --- the original series, without the normalizing factors --- expressed as a power series in $X_i=q^{-s_i}$.
\end{rem*}
\begin{proof}

Let $\mathcal{F}(\X)=\mathcal{N}(\X)/D(\x)$ be any rational function with the same polar behavior as $\mathcal{Z}^\ast(\X;\bm)$ and that satisfies \eqref{eqn:globfe2}. Write $\mathcal{N}(\X)=\sum_{\lambda} b_\lambda \X^\lambda$. By the proof of Proposition \ref{prop:localglobal} and Theorem \ref{thm:unstab}, the polynomial $\mathcal{N}(\X)$ is completely determined by the coefficients $b_{\theta-\xi}$ for $\xi \in \Theta^+$.

Since both $\mathcal{F}(\X)$ and $\tilde{F}(\x;\xi)$ share the same denominator, Theorem \ref{thm:unstabsum} yields
\begin{equation}
\label{eqn:sum}
\mathcal{F}(\X)=\sum_{\xi \in \Theta^+} m_{\xi}\tilde{F}(\X;\xi)\X^{\theta-\xi},\hspace{1cm} m_\xi \in \C.
\end{equation}

It remains to identify the coefficients $M_\lambda$. For this we let $f(\x;\ell):=\Delta(\x)F(\x;\ell)$. Let $m_\xi=M_{\theta-\xi}$, where $M_\lambda$ is the $\X^\lambda$ coefficient of $Z(\bs;\bm)$, written as a power series in the $X_i=q^{-s_i}$. For each $\xi \in \Theta^+$, we have $\tilde{F}(\X;\xi)=\tilde{f}(\X;\xi)/\tilde{\Delta}(\X)$. It follows that
 \begin{equation}
 \label{eqn:Fs}
 \sum_{\xi \in \Theta^+} m_\xi \tilde{F}(\X;\xi)\X^{\theta-\xi}=\frac{1}{\tilde{\Delta}(\X)}\sum_{\xi \in \Theta^+} m_\xi \tilde{f}(\X;\xi)\X^{\theta-\xi}.
 \end{equation}
Since $\mathcal{Z}^\ast(\X;\bm)=\Xi(\bs)Z(\bs;\bm)$ and $\Xi(\bs)=\tilde{\Delta}(\X)^{-1}$, from \eqref{eqn:Fs} we have 
\[
Z(\bs;\bm):=\sum_{\lambda} M_\lambda \x^\lambda=\sum_{\xi \in \Theta^+} m_\xi \tilde{f}(\X;\xi)\X^{\theta-\xi}.
\]
Writing both sides as power series in the $X_i$, the constant coefficient of every series on the right-hand side is one. To equate $m_\xi$ with $M_{\theta-\xi}$, we induct on $\xi$. It is clear that $m_\theta=M_0=1$. Choose $\xi\prec \theta$, and suppose that $m_\xi'=M_{\theta-\xi'}$ for all $\xi'\succ\xi$. By Theorem \ref{thm:gap}, the $\tilde{f}(\X;\xi')\X^{\theta-\xi'}$ are supported outside or on the boundary of the $\Pi_\xi'$. Thus scaling $\tilde{f}(\X;\xi')\X^{\theta-\xi'}$ by $M_{\theta-\xi'}$ does not affect the coefficients of $\tilde{f}(\X;\xi)\X^{\theta-\xi}$. Since the constant coefficient of $\tilde{f}(\X;\xi)\X^{\theta-\xi}$ is one, we must have $m_\xi=M_{\theta-\xi'}$.  
\end{proof}

\subsection{Examples}
\label{sec:examples}
We now apply Theorem \ref{thm:localglobal} to two low rank examples. First we consider $\Phi=A_2$. We take $n=3$ and $q=7$, and choose $\bm=(T^3+5T+2,1)$. Note that $T^3+5T+2$ is irreducible over $\F_7$. Put $\theta=4\varpi_1+\varpi_2$. One can show  $\Theta^+=\{\theta,2(\varpi_1+\varpi_2),\varpi_1+\varpi_2\}$ and $\{\theta-\xi:\xi\in \Theta^+\}=\{0,\alpha_1,2\alpha_1+\alpha_2\}.$ Let $\X=(X_1,X_2)$ where $X_i=q^{-s_i}$. Theorem \ref{thm:localglobal} implies that
     \begin{align}
     \label{eqn:applythm}
	\mathcal{Z}^\ast(\X;\bm)&=M_{0}\tilde{F}(\X;3,0)+M_{\alpha_1}\tilde{F}(\X;1,1)X_1+M_{2\alpha_1+\alpha_2}\tilde{F}(\X;0,0)X_1^2X_2,
     \end{align}
    where
    \begin{align*}
    M_{0}&=H(1,1;\bm)=1,\\
  M_{\alpha_1}&=\sum_{c_1 \deg 1} H(c_1,1;\bm),\\
     M_{2\alpha_1+\alpha_2}&=\sum_{c_1 \deg 2}\sum_{c_2 \deg 1} H(c_1,c_2;\bm).
     \end{align*}
	Using Magma, we find that $M_{\alpha_1}=\tau(\epsilon)(-0.5+2.598i)$ and  $M_{2\alpha+\alpha_2}\approx \tau(\epsilon)^3(6.5+2.598i)$. Figure \ref{fig:globsupport1} shows the support of the numerator of the global series $Z^\ast(\X;\bm)$, expressed in terms of the shifted $p$-parts. The support of $\tilde{N}(\X;3,0)$ is shown in light gray, the support of $M_{\alpha_1}\tilde{N}(\X;1,1)X_1$ is shown in gray, and the support of $M_{2\alpha_1+\alpha_2}\tilde{N}(\X;0,0)X_1^2X_2$ is shown in black.
	\begin{rem*}
	To simplify the computation of the Gauss sums, we used the following useful fact: For $c \in A$, let $\mu(c)$ denote the M\"{o}bius function. Then \cite[Theorem 2.1]{patterson}  \[g(1,c)=\mu(c)\left(\leg{c'}{c}_3\right)(-\tau(\epsilon))^{\deg{c}}.\]
	\end{rem*}

	\begin{rem*}
	We note that $M_{\alpha_1}$ and $M_{2\alpha_1+\alpha_2}$ are algebraic. In fact, they live in the compositum of $\Q(\zeta_3)$ and $\Q(\zeta_7+\zeta_7^{-1})$, where $\zeta_a$ denotes a primitive $a$th root of unity.
	\end{rem*}

%

\begin{figure}[h!]
\includegraphics[scale=.55]{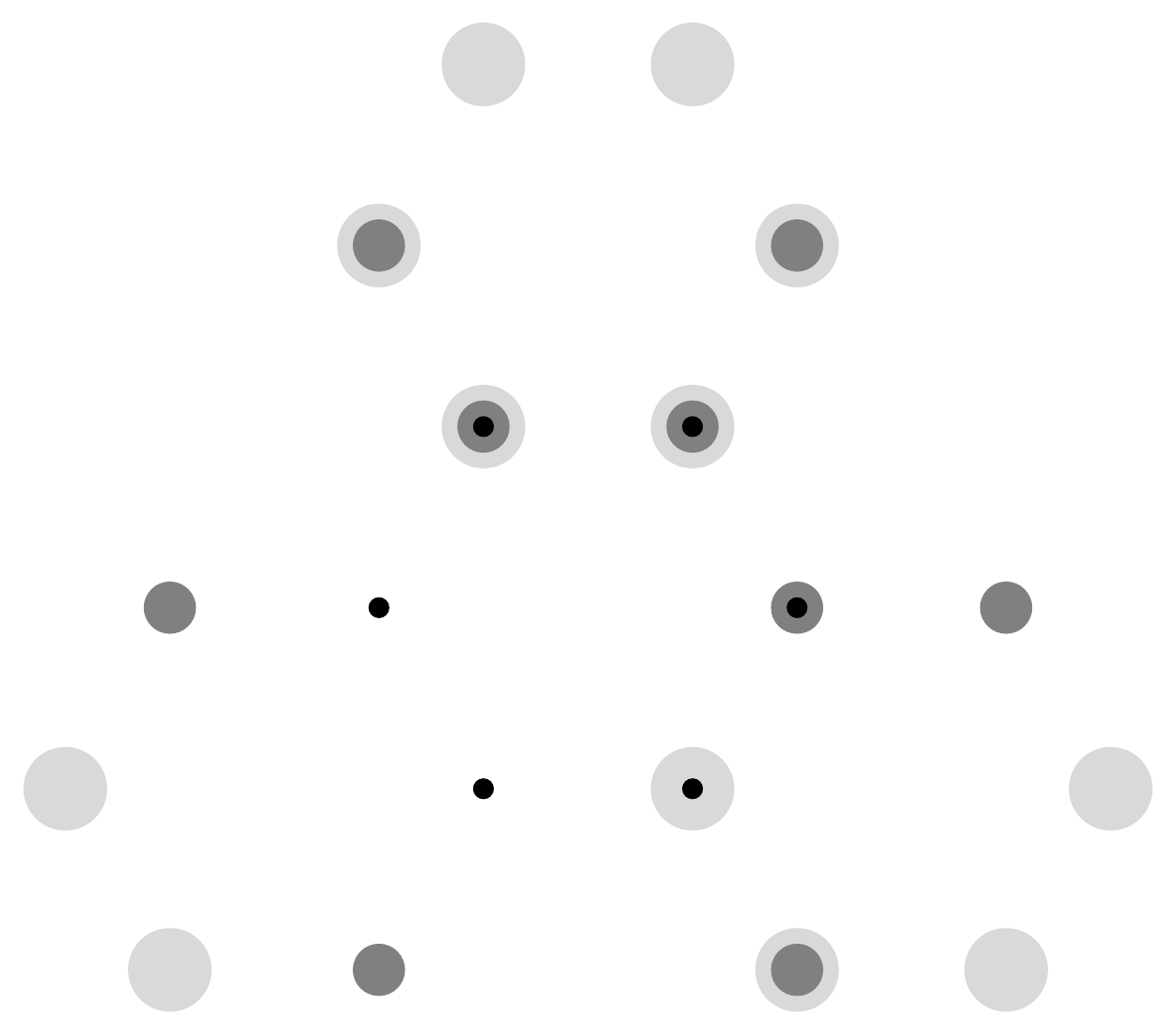}
\caption{Support of the numerator of $Z_{A_2}^\ast(\X;T^3+5T+2,1)$, in terms of the $\tilde{N}(\X;\xi)$.}
\label{fig:globsupport1}
\end{figure}

Next we consider $\Phi=B_2$. We take $n=2$ and $q=5$, and choose ${\bf m}=(1,T^2+2)$. To apply Theorem, we put $\theta=\varpi_1+3\varpi_2$. One can show that 
$\Theta^+=\{\rho, \varpi_1+3\varpi_2,2\varpi_1+\varpi_2\}$ and $\{\theta-\xi:\xi \in \Theta^+\}=\{0,\alpha_2,\alpha_1+2\alpha_2\}.$
Thus
\[Z^\ast(\X;\bm)=M_{0}\tilde{F}(\X;\varpi_1+3\varpi_2)
+M_{\alpha_2}\tilde{F}(\X;2\varpi_1+\varpi_2)X_2
+M_{\alpha_1+2\alpha_2}\tilde{F}(\X;\rho)X_1X_2^2,
\]
where
    \begin{align*}
    M_{0}&=H(1,1;\bm)=1,\\
  M_{\alpha_2}&=\sum_{c_2 \deg 1} H(1,c_2;\bm),\\
     M_{\alpha_1+2\alpha_2}&=\sum_{c_1 \deg 1}\sum_{c_2 \deg 2} H(c_1,c_2;\bm).
     \end{align*}

Using Magma, we compute that $M_0=1$, $M_{\alpha_2}=-\tau(\epsilon)$, and $M_{\alpha_1+2\alpha_2}=q^2$. Figure \ref{fig:globsupport2} shows the support of the numerator of the global series $Z^\ast(\X;\bm)$, expressed in terms of the shifted $p$-parts. The support of $\tilde{N}(\X;0,2)$ is shown in light gray, the support of $M_{\alpha_1}\tilde{N}(\X;1,0)X_2$ is shown in gray, and the support of $M_{\alpha_1+2\alpha_2}\tilde{N}(\X;0,1)X_1X_2^2$ is shown in black.

\begin{figure}[h!]
\includegraphics[scale=.55]{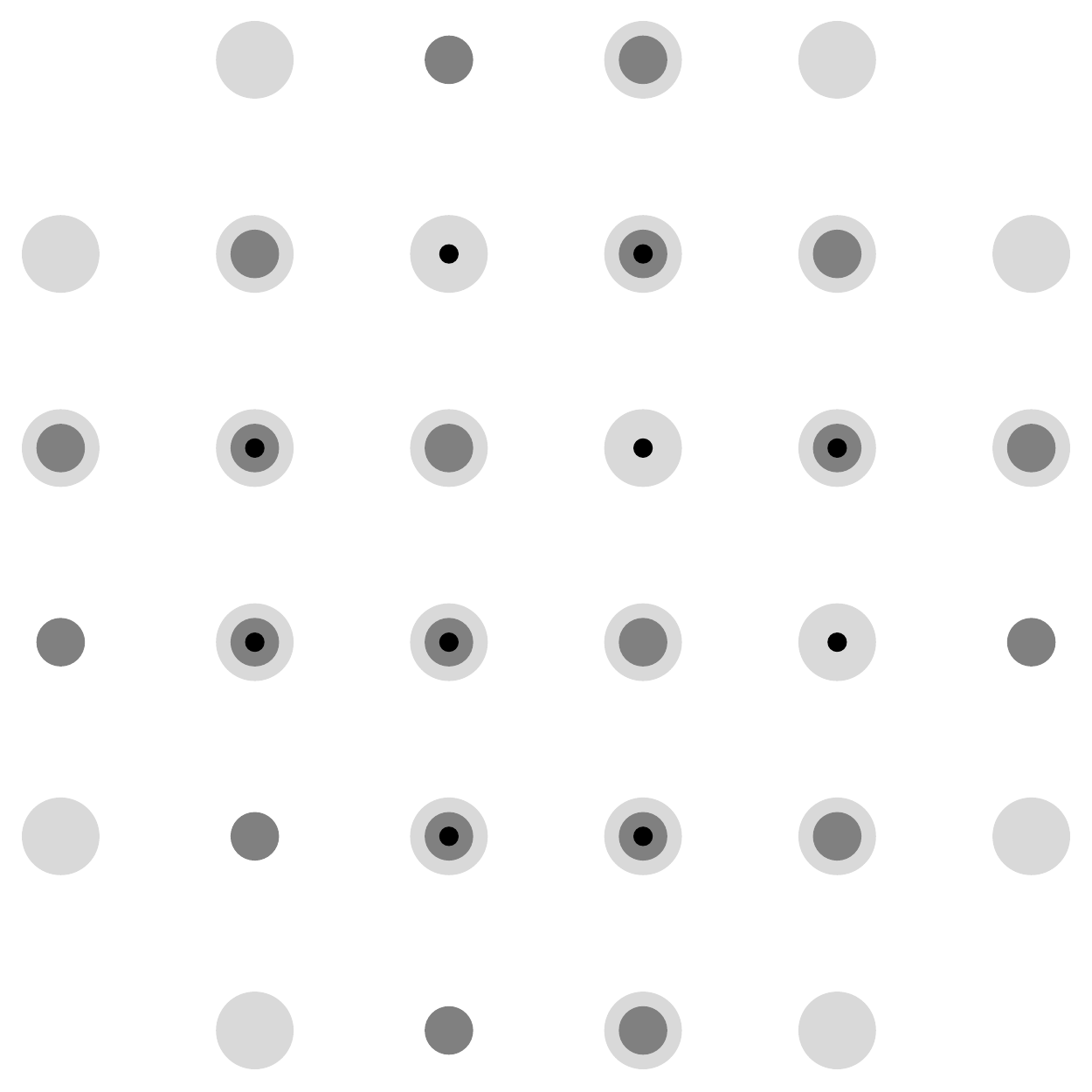}
\caption{Support of the numerator of $Z_{B_2}^\ast(\X;1,T^2+2)$, in terms of the $\tilde{N}(\X;\xi)$.}
\label{fig:globsupport2}
\end{figure}

\bibliographystyle{plain}
\bibliography{twistref}

\end{document}